\documentclass[11pt]{amsart}
\usepackage{eurosym}
\usepackage[pagebackref,colorlinks=true]{hyperref}
\usepackage{verbatim}
\usepackage{eucal,url,amssymb,stmaryrd,enumerate,amscd,}
\usepackage{showkeys}
\usepackage{amsfonts}
\usepackage{amsmath,amsthm,amssymb,amscd,enumerate,eucal,url,stmaryrd}
\usepackage[margin=1in]{geometry}

\numberwithin{equation}{section}
\newtheorem{thrm}{Theorem}[section]

\newtheorem{cor}[thrm]{Corollary}
\newtheorem{dfn}[thrm]{Definition}
\newtheorem{rmrk}[thrm]{Remark}

\newtheorem{conv}[thrm]{Convention}
  \allowdisplaybreaks

\begin{document}

\begin{abstract}
Connections with (skew-symmetric) torsion on non-symmetric
Riemannian manifold satisfying the Einstein metricity condition
(NGT with torsion) are considered. It is shown that an almost
Hermitian manifold is an NGT with torsion if and only if it is a
Nearly K\"ahler manifold. In the case of an almost contact metric
manifold the NGT with torsion spaces are characterized and a
possibly new class of almost contact metric manifolds is
extracted. Similar considerations lead to a definition of a
particular  classes of almost para-Hermitian and almost
paracontact metric manifolds. The conditions are given in terms of
the corresponding Nijenhuis tensors and the exterior derivative of
the skew-symmetric part of the non-symmetric Riemannian metric.
\end{abstract}

\title[Connections on non-symmetric (generalized) Riemannian manifold and gravity]{Connections on non-symmetric (generalized)
Riemannian manifold and gravity}
\date{\today}
\author{Stefan Ivanov}
\address[Ivanov]{University of Sofia "St. Kl. Ohridski"\\
Faculty of Mathematics and Informatics\\
Blvd. James Bourchier 5\\
1164 Sofia, Bulgaria; Institute of Mathematics and Informatics,
Bulgarian Academy of Sciences}
\address{and Department of Mathematics,
University of Pennsylvania, DRL 209 South 33rd Street
Philadelphia, PA 19104-6395 } \email{ivanovsp@fmi.uni-sofia.bg}
\author{Milan Zlatanovi\'c}
\address[Zlatanovi\'c]{Department of Mathematics, Faculty of Science and Mathematics, University of Ni\v s, Vi\v segradska 33, 18000 Ni\v s, Serbia}
\email{zlatmilan@yahoo.com}

\date{\today }
\maketitle \tableofcontents


\setcounter{tocdepth}{2}

\section{Introduction}

In this note we consider connections on a non-symmetric Riemannian
manifold $(M, G=g+F)$ with a general assumption, coming from the
non-symmetric gravitational theory, that the symmetric part $g$ of
$G$ is non-degenerate. There are many examples of generalized
Riemannian manifolds among which it deserves to mention the almost
Hermitian, the almost contact, the almost para-Hermitian, the
almost paracontact manifolds etc.

It is happened that the torsion and the covariant derivative of
the non-symmetric metric determine a linear connection. The
connection is unique and it is completely determined by the
torsion and the Levi-Civita connection of the symmetric part of
the non-symmetric metric. The Levi-Civita covariant derivative of
the skew-symmetric part of the non-symmetric metric and the
torsion supply  a necessary condition for the existence of the
connection.

\subsection{Motivation from general relativity}

General relativity (GR) was developed by A. Einstein in 1916
\cite{Ein1}, with contributions by many others after 1916. In GR
the equation $ds^2=g_{ij}dx^idx^j, \quad (g_{ij}=g_{ji})$
  is valid, where
$g_{ij}$ are functions of the point in the space. In GR which is the
four dimensional space-time continuum metric properties depend on
the mass distribution.
 The magnitudes $g_{ij}$  are known as \emph{gravitational
 potential}.
Christoffel symbols which are commonly expressed by
$\Gamma^k_{ij}$, play the role of the magnitudes which determine
gravitational force field. General relativity explains gravity as
the curvature of space-time.

In the GR the metric tensor  is related by the Einstein equations
$
R_{ij}-\frac 12 R g_{ij}=T_{ij}, $
where $R_{ij}$ is the Ricci tensor of metric of space time, $R$ is
the scalar curvature of the metric, and $T_{ij}$ is the
energy-momentum tensor of matter. In 1922, A. Friedmann
\cite{Fried} found a solution in which the universe may expand or
contract, and later G. Lema\^{\i}tre \cite{Lema} derived a
solution for an expanding universe. However, Einstein believed
that the universe was apparently static, and since a static
cosmology was not supported by the general relativistic field
equations, he added a cosmological constant $\Lambda$ to the field
equations, which became
$
R_{ij}-\frac 12 R g_{ij}+\Lambda g_{ij} =T_{ij}. $

Since 1923 to the end of his life Einstein worked on various
variants of Unified Field Theory (Non-symmetric Gravitational
Theory-NGT) \cite{Ein}. This theory had the aim to unite the
gravitation theory, to which is related GR, and the theory of
electromagnetism. Introducing different variants of his NGT,
Einstein used a complex basic tensor, with symmetric real part and
skew-symmetric imaginary part. Beginning with 1950. Einstein used
real non-symmetric basic tensor $G$, sometimes called
\emph{generalized Riemannian metric/manifold}).

Remark that at NGT the symmetric part $g_{ij}$ of the basic tensor
$G_{ij} (G_{ij}=g_{ij}+F_{ij})$ is related to gravitation, and
skew-symmetric one $F_{ij}$ to electromagnetism. The same is valid
for the symmetric part of the connection and torsion tensor,
respectively.

More recently the idea of a non-symmetric metric tensor appears in
Moffat's non-symmetric gravitational theory \cite{Mof}.
In Moffat's theory the skew-symmetric part of the metric tensor
represents a Proca field (massive Maxwell field) which is a part
of the gravitational interaction, contributing to the rotation of
galaxies.

While on a Riemannian space  the connection coefficients are
expressed by virtue of the metric, $g_{ij}$, at Einstein's works
on NGT the connection between these magnitudes is determined by
the so called \emph{Einstein metricity condition} i.e. the
non-symmetric metric tensor $G$ and the connection components
$\Gamma_{ij}^k$ are connected with the equations
\begin{equation}\label{metein}
\frac{\partial G_{ij}}{\partial x^m}-\Gamma ^p_{im}G_{pj}- \Gamma
^p_{mj}G_{ip}=0.
\end{equation}
A generalized Riemannian manifold satisfying the Einstein
metricity condition \eqref{metein} is also called NGT-space
\cite{Ein,LMof,Mof}.

The choice of a connection in NGT is not uniquely determined. In
particular, in NGT there exist two kinds of covariant derivative,
for example for tensor $a^i_j:$
\begin{equation*}
 a{}^{\underset+i}_{\underset+j|m}=\frac{\partial a^i_{j}}{\partial x^m}
 +\Gamma^i_{pm}a^p_j-\Gamma^p_{jm}a^i_p; \quad
 a{}^{\underset-i}_{\underset-j|m}=\frac{\partial a^i_{j}}{\partial
 x^m}+\Gamma^i_{mp}a^p_j-\Gamma^p_{mj}a^i_p,
 \end{equation*}

We investigate  connections on a generalized Riemannian manifold
$(M, G=g+F)$ with a general assumption, coming from NGT, that the
symmetric part $g$ of $G$ is non-degenerate. There are many
examples of generalized Riemannian manifolds such as  the almost
Hermitian, the almost contact, the almost para-Hermitian, the
almost paracontact manifolds etc. In these cases the
skew-symmetric part of $G$ is played by the fundamental 2-form
$F$.

We show that the torsion and the covariant derivative of $G$
determine  a unique linear connection $\nabla$ with given torsion
$T$ and covariant derivative $\nabla G$ provided the covariant
derivative of the skew-symmetric part $F$ of $G$, $\nabla F$
satisfies certain compatibility condition which we expressed in
terms of the torsion and the covariant derivative of the symmetric
part $g$ of $G$, $\nabla g$ (Theorem~\ref{gmain}). The connection
is unique and it is completely determined from the symmetric part
$g$ of $G$.

We look for a linear connections with torsion preserving the
non-symmetric tensor $G$.  Surprisingly, we find that a linear
connection $\nabla$ preserves the non-symmetric metric, $\nabla
G=0$ if and only if it preserves its symmetric and skew-symmetric
parts, $\nabla g=\nabla F=0$. The last condition leads to a
necessary constrain, i.e. a PDE for the skew-symmetric part $F$ of
$G$ expressed in terms of the torsion. We show that there exists a
unique connection preserving $G$ provided its torsion is given and
the mentioned above PDE is satisfied (c.f. Theorem~\ref{main}).

We paid a special attention when the torsion of the connection
preserving the generalized Riemannian metric $G$ is totally
skew-symmetric with respect to the symmetric part $g$ of $G$. One
reason for that comes from the supersymmetric string theories and
non-liner $\sigma$-models (see e.g.
\cite{Str,BBDG,BBE,BBDP,y4,BSethi,GKMW,GMPW,GMW,GKP,GPap} and
references therein) as well as from gravity theory itself
\cite{Ham}. We find a condition in terms of the Nijenhuis tensor
and the exterior derivative of the skew-symmetric part $F$ of $G$
and show that this condition is necessary and sufficient for the
existence of connection with skew-symmetric torsion preserving the
generalized Rieamannian metric $G$ (c.f. Theorem~\ref{skew}). We
note that such a conditions were previously known on almost
Hermitian, almost contact, almost para-Hermitian, almost
paracontact manifolds \cite{FI,GKMW,Zam}. We derive the results
there as a consequence of our considerations.

In Section~\ref{ngt} we restrict our considerations to the NGT,
i.e. a generalized Riemannian manifold satisfying the Einstein
metricity condition \eqref{metein}. We present a condition in
terms of the Nijenhuis tensor and the exterior derivative of $F$
which guarantees the existence of a unique linear connection with
skew-symmetric torsion preserving the NGT structure and show that
the torsion is equal to minus one third of the exterior derivative
$dF$ of $F$ (Theorem~\ref{pp1}).

One of the main contributions is the application of this result to
the almost Hermitian and almost contact metric manifolds.

A careful analysis of our general condition for the existence of a
connection with skew-symmetric torsion satisfying the Einstein
metricity condition in the case of almost Hermitian manifold allow
us to conclude that an almost Hermitian manifold satisfies the
Einstein metricity condition with respect to a connection with
skew-symmetric torsion if and only if it is a Nearly K\"ahler
manifold (Theorem~\ref{nika}). In other words, an almost Hermitian
manifolds is NGT with slew-symmetric torsion exactly when it is a
Nearly K\"ahler manifold. In this case the connection coincides
with the Gray connection \cite{Gr1,Gr2,Gr3} which is the unique
connection with skew-symmetric torsion preserving the Nearly
K\"ahler structure \cite{FI}. The  Nearly K\"ahler manifolds
(called almost Tachibana spaces in \cite{Yano}) were developed by
A. Gray \cite{Gr1,Gr2,Gr3} and intensively studied since then in
\cite{Ki,N1,MNS,N2,But,FIMU}. Nearly K\"ahler manifolds appear
also in supersymmetric string theories (see e.g.
\cite{Po1,PS,LNP,GLNP}. The first complete inhomogeneous examples
of 6-dimensional Nearly K\"ahler manifolds were presented
recently in \cite{FosH}.

Applying the general condition for the existence of a connection
with skew-symmetric torsion satisfying the Einstein metricity
condition to the almost contact metric manifold we arrived at a
possibly new class of almost contact metric manifolds, which we
call here \emph{almost-nearly cosymplectic}. We show that this is
the precise class of almost contact metric manifold which are NGT
with skew-symmetric torsion, i.e. admitting a connection with
skew-symmetric torsion satisfying the Einstein metricity
condition. The simplest example of almost-nearly cosymplectic
manifold is the trivial circle bundle over a (compact) nearly
K\"ahler manifold.

The cases of almost para-Hermitian and almost paracontact metric
manifold are also treated with respect to the NGT space with
skew-symmetric torsion. We show that an almost para-Hermitian
manifold admits an NGT connection with totally skew symmetric
torsion if and only if it admits a linear connection preserving
the almost para-Hermitian structure with a totally skew-symmetric
torsion. We extract a possibly new class of almost paracontact
metric structures which are NGT with skew-symmetric torsion. We
characterized this class by an explicit formula for the covariant
derivative with respect to the Levi-Civita connection of the
fundamental two form.

\section{The geometric model}
\label{geomod}

The fundamental (0,2) tensor  $G$ in non-symmetric (generalized)
Riemannian manifold $(M,G)$ is in general non-symmetric. It is
decomposed in two parts, the symmetric part $g$ and the
skew-symmetric part $F$, $G(X,Y)=g(X,Y)+F(X,Y),$ where
\begin{equation}\label{metric}
g(X,Y)=\frac12(G(X,Y)+G(Y,X)), \qquad F(X,Y)=\frac12(G(X,Y)-G(Y,X)).
 \end{equation}
We assume that the symmetric part is non-degenerate of arbitrary
signature. Therefore, we obtain a well defined (1,1) tensor $A$
determined by the condition
\begin{equation}\label{m1}
F(X,Y)=g(AX,Y).
\end{equation}

We look for a natural linear connections $\nabla$ preserving the
generalized Riemannian metric $G, \nabla G=0$  with torsion
$T(X,Y)=\nabla_XY-\nabla_YX-[X,Y]$.

\begin{conv}
In the whole paper we shall use the capital Latin letters $X,Y,..$
to denote smooth vector fields on a smooth manifold $M$ which
commute, $[X,Y]=0$. Hence, $T(X,Y)=\nabla_XY-\nabla_YX$.
\end{conv}

The Levi-Civita connection corresponding to the symmetric
non-degenerate (0,2) tensor $g$ we denote with $\nabla^g$. The
Koszul formula reads
\begin{equation}\label{lcg}
g(\nabla^g_XY,Z)=\frac12\big[Xg(Y,Z)+Yg(X,Z)-Zg(Y,X).\Big]
\end{equation}
We denote the (0,3) torsion tensor with respect to $g$ by the same
letter, $
T(X,Y,Z):=g(T(X,Y),Z). $
\subsection{Linear connections on generalized Riemannian manifolds}
In this section we show that a linear connection $\nabla$  with
torsion tensor $T$ on a generalized Riemannian manifold is
completely determined by the torsion and the covariant derivative
$\nabla g$ of the symmetric part  $g$ of $G$. More precisely, we
have
\begin{thrm}\label{gmain}
Let $(M,G=g+F)$ be a generalized Riemannian manifold and
$\nabla^g$ be the Levi-Civita connection of $g$. Let $\nabla$ be a
linear connection with torsion $T$ and denote the covariant
derivative of the symmetric part $g$ of $G$ by $\nabla g$. Then
$\nabla$ is unique determined by the following formula
\begin{multline}\label{cgencon}
g(\nabla_XY,Z)=g(\nabla^g_XY,Z)+\frac12\Big[T(X,Y,Z)+T(Z,X,Y)-T(Y,Z,X)\Big]\\
-\frac12\Big[(\nabla_Xg)(Y,Z)+(\nabla_Yg)(Z,X)-(\nabla_Zg)(Y,X)\Big].
\end{multline}
The covariant derivative $\nabla F$ of the skew-symmetric part $F$
of $G$ is given by
\begin{multline}\label{cconn2}
(\nabla_XF)(Y,Z)=(\nabla^g_XF)(Y,Z)+\frac12\Big[T(X,Y,AZ)+T(Z,X,AY)\Big]\\+
\frac12\Big[T(AZ,X,Y)+T(AZ,Y,X)+T(X,AY,Z)+T(Z,AY,X)\Big]\\+
\frac12\Big[(\nabla_Xg)(AY,Z)-(\nabla_Xg)(Y,AZ)-(\nabla_Yg)(AZ,X)\Big]\\+\frac12\Big[(\nabla_Zg)(AY,X)+(\nabla_{AZ}g)(Y,X)-(\nabla_{AY}g)(Z,X)\Big].
\end{multline}
In particular, the exterior derivative $dF$ of $F$ satisfies
\begin{multline}\label{cconn1}
dF(X,Y,Z)
=-T(X,Y,AZ)-T(Y,Z,AX)-T(Z,X,AY)\\+(\nabla_XF)(Y,Z)+(\nabla_YF)(Z,X)+(\nabla_ZF)(X,Y).
\end{multline}
Conversely, any three tensors $T,\nabla g,\nabla F$ satisfying
\eqref{cconn2} determine a unigue linear connection via
\eqref{cgencon}.
\end{thrm}
\begin{proof}
It follows from \eqref{m1} that the (1,1) tensor $A$ is
skew-symmetric with respect to the pseudo-Riemannian metric $g$,
$g(AX,Y)=-g(X,AY).$
A simple calculation  using \eqref{m1} yields
\begin{multline}\label{a2}
(\nabla_XF)(Y,Z)=Xg(AY,Z)-g(A\nabla_XY,Z)-g(AY,\nabla_XZ)\\=(\nabla_Xg)(AY,Z)+g((\nabla_XA)Y,Z).
\end{multline}
From the definition of the covariant derivative of a (0,2) tensor,
we have
\begin{equation}\label{covG}
\begin{aligned}
(\nabla_XG)(Y,Z)=XG(Y,Z)-G(\nabla_XY,Z)-G(Y,\nabla_XZ);\\
(\nabla_YG)(Z,X)=YG(Z,X)-G(\nabla_YZ,X)-G(Z,\nabla_YX);\\
(\nabla_ZG)(Y,X)=ZG(Y,X)-G(\nabla_ZY,X)-G(Y,\nabla_ZX).
\end{aligned}
\end{equation}
Sum the first two equalities and subtract the third equality to get
using \eqref{metric} and the definition of the torsion that
\begin{multline}\label{ccon1}
(\nabla_XG)(Y,Z)+(\nabla_YG)(Z,X)-(\nabla_ZG)(Y,X)=
Xg(Y,Z)+Yg(Z,X)-Zg(Y,X)\\+dF(X,Y,Z)-2g(\nabla_XY,Z)+G(Z,T(X,Y))-G(Y,T(X,Z))-G(T(Y,Z),X).
\end{multline}
Using the Koszul formula \eqref{lcg}, we obtain from \eqref{ccon1}
that
\begin{multline}\label{ccon2}
2g(\nabla_XY,Z)=2g(\nabla^g_XY,Z) +
dF(X,Y,Z)+G(Z,T(X,Y))\\-G(Y,T(X,Z))-G(T(Y,Z),X)
-(\nabla_XG)(Y,Z)-(\nabla_YG)(Z,X)+(\nabla_ZG)(Y,X).
\end{multline}
The skew-symmetric part with respect to $X,Y$ of \eqref{ccon2} gives
precisely the equation \eqref{cconn1}.

Substitute \eqref{cconn1} into \eqref{ccon2} and use
\eqref{metric} to get \eqref{cgencon}. Insert the already prooved
\eqref{cgencon} into
$$(\nabla_XF)(Y,Z)=XF(Y,Z)-F(\nabla_XY,Z)-F(Y,\nabla_XZ),$$  use
\eqref{m1} and \eqref{a2} to obtain \eqref{cconn2}.

To complete the proof, observe that \eqref{cconn2} implies
\eqref{cconn1}. Indeed, taking the cyclic sum of \eqref{cconn2} we
get \eqref{cconn1}.

The converse follows from \eqref{cgencon} by straightforward
computation.
\end{proof}

\subsection{Metric connections on generalized Riemannian manifold} Here we investigate the existence of
metric connections on a generalized Riemannian manifold. We have
\begin{thrm}\label{main}
Let $(M,G=g+F)$ be a generalized Riemannian manifold and
$\nabla^g$ be the Levi-Civita connection of $g$.
\begin{enumerate}
\item[a)] A linear connection $\nabla$ preserves the generalized
Riemannian metric $G$ if and only if it  preserves its symmetric
part $g$ and its skew-symmetric part $F$,
$
\nabla G=0 \Leftrightarrow \nabla g=\nabla F=0 \Leftrightarrow
\nabla g=\nabla A=0.
$ 
\item[b)] If there exists a linear connection $\nabla$ preserving the  generalized Riemannian metric $G, \nabla G=0$
with torsion $T$ then  the following  condition holds
\begin{multline}\label{conn2}
(\nabla^g_XF)(Y,Z)=-\frac12\Big[T(X,Y,AZ)+T(Z,X,AY)\Big]\\-
\frac12\Big[T(AZ,X,Y)+T(AZ,Y,X)+T(X,AY,Z)+T(Z,AY,X)\Big].
\end{multline}
In particular, the exterior derivative of $F$ satisfies the next
equality
\begin{multline}\label{conn1}
dF(X,Y,Z)=F(T(X,Y),Z)+F(T(Y,Z),X)+F(T(Z,X),Y),\quad \rm equivalently\\
dF(X,Y,Z)=-T(X,Y,AZ)-T(Y,Z,AX)-T(Z,X,AY);
\end{multline}
\indent Conversely, if the condition \eqref{conn2} is valid then
there exists a unique linear connection $\nabla$ with torsion $T$
preserving the generalized Riemannian metric $G$ 
determined by the torsion T   with the  formula
\begin{equation}\label{gencon}
g(\nabla_XY,Z)=g(\nabla^g_XY,Z)+\frac12\Big[T(X,Y,Z)+T(Z,X,Y)-T(Y,Z,X)\Big].
\end{equation}
\end{enumerate}
\end{thrm}
\begin{proof}
Suppose we have $0=\nabla G=\nabla g + \nabla F$. The cyclic sum of
this equality yields
\begin{multline*}
(\nabla_Xg)(Y,Z)+(\nabla_Yg)(Z,X)+(\nabla_Zg)(X,Y)\\=-(\nabla_XF)(Y,Z)-(\nabla_YF)(Z,X)-(\nabla_ZF)(X,Y).
\end{multline*}
Observe that the left hand side is totally symmetric while the
right hand side is totally skew-symmetric which leads to the
vanishing of each side,
\begin{equation}\label{new1}
\begin{split}
(\nabla_Xg)(Y,Z)+(\nabla_Yg)(Z,X)+(\nabla_Zg)(X,Y)=0;\\(\nabla_XF)(Y,Z)+(\nabla_YF)(Z,X)+(\nabla_ZF)(X,Y)=0.
\end{split}
\end{equation}
On the other hand, the symmetric part of $\nabla G$ with respect to
the first two arguments gives
$$(\nabla_Xg)(Y,Z)+(\nabla_Yg)(X,Z)+(\nabla_XF)(Y,Z)+(\nabla_YF)(X,Z)=0$$
which combined with the first identity in \eqref{new1} yields
\begin{equation}\label{new2}(\nabla_Zg)(X,Y)=(\nabla_XF)(Y,Z)+(\nabla_YF)(X,Z).\end{equation}
Substitute \eqref{new2} into  $\nabla G=\nabla g+\nabla F=0$ to
get
$$(\nabla_XF)(Y,Z)+(\nabla_YF)(X,Z)+(\nabla_ZF)(X,Y)=0$$
which combined with the second equality in \eqref{new1} gives
$2(\nabla_YF)(X,Z)=0$. Now  $\nabla g=0$ follows from
\eqref{new2}. Clearly, the conditions $\nabla g=\nabla F=0$ imply
$\nabla G=0$. We conclude from \eqref{a2} that a metric connection
$\nabla, \nabla g=0$ preserves the two form $F$ if and only if it
preserves the (1,1) tensor $A$, $\nabla F=0\Leftrightarrow\nabla
A=0$ since $g$ is non-degenerate. This proofs a).

Set $\nabla g=\nabla F=0$ into Theorem~\ref{gmain} to conclude b).
Apply Theorem~\ref{gmain} to complete the proof.
\end{proof}
\subsection{The torsion tensor, skew-symmetric torsion} It follows from Theorem~\ref{main} that any linear connection
preserving the generalized Riemannian metric is completely
determined by the torsion tensor $T$. In this section we
investigate the torsion tensor. To this end we recall the
definition of the Nijenhuis tensor $N$ of the (1,1) tensor $A$
(see e.g.\cite{KN}),
\begin{equation}\label{nuj}
N(X,Y)=[AX,AY]+A^2[X,Y]-A[AX,Y]-A[X,AY].
\end{equation}
The Nijenhuis tensor is skew-symmetric by definition and it plays
a fundamental role in almost complex (resp. almost para-complex)
geometry. If $A^2=-1$ (resp. $A^2=1$) then the celebrated
Nulander-Nirenberg theorem (see, e.g. \cite{KN}) shows that an
almost complex structure  is integrable if and only if the
Nijenhuis tensor vanishes.

Let $\nabla$ be a linear connection preserving the generalized
Riemannian metric $G, \nabla G=0$. The $\nabla$ preserves $g,F$ and
$A$, $\nabla g=\nabla F=\nabla A=0$. Using the definition of the
torsion and the covariant derivative  $\nabla A$ we  can express the
Nijenhuis tensor of terms of the torsion and $\nabla A$ as follows
\begin{multline}\label{nuj1}
N(X,Y)=(\nabla_{AX}A)Y-(\nabla_{AY}A)X-A(\nabla_{X}A)Y+A(\nabla_{Y}A)X\\-T(AX,AY)-A^2T(X,Y)+AT(AX,Y)+AT(X,AY).
\end{multline}
We denote the Nijenhuis tensor of type (0,3) with respect to $g$
with the same letter, $N(X,Y,Z):=g(N(X,Y),Z)$. Set $\nabla A=0$
into \eqref{nuj1} and use \eqref{m1} to get
\begin{equation}\label{nuj2}
N(X,Y,Z)=-T(AX,AY,Z)-T(X,Y,A^2Z)-T(AX,Y,AZ)-T(X,AY,AZ).
\end{equation}
\subsubsection{The skew-symmetric torsion} Linear connections
preserving a (pseudo) Riemannian metric and having totally
skew-symmetric torsion i.e. the torsion satisfies the condition
\begin{equation}\label{st}
T(X,Y,Z)=-T(X,Z,Y),
\end{equation}
become very attractive in the last twenty years mainly due to the
relations with supersymmetric string theories (see e.g.
\cite{Str,GMW,GPap,y4,BSethi} and references therein, for a
mathematical treatment  consult the nice overview \cite{Agr}). The
main point is that the number of preserved supersymmetries is
equal to the number of parallel spinors with respect to such a
connection. This property reduces the holonomy group of the
connection to be a subgroup of a group which is stabilizer of a
non-trivial spinor such as  $SU(n), Sp(n), G_2, Spin(7)$. It is
happened that in that cases such a connection is unique
\cite{Str,GKMW,FI,Iv} determined entirely by the structure induced
by the parallel spinor.

We investigate when a generalized Riemannian manifold admits a
metric connection with skew-symmetric torsion. We have

\begin{thrm}\label{skew}
 Let $(M,G=g+F)$ be a generalized Riemannian manifold and
$\nabla^g$ be the Levi-Civita connection of $g$. If there exists a
linear connection $\nabla$ preserving the generalized Riemannian
metric $G, \nabla G=0$ with totally skew-symmetric torsion $T$
then the following condition holds
\begin{equation}\label{ndf1}
N(X,Y,AZ)+N(X,Z,AY)=dF(X,Y,A^2Z)+dF(X,Z,A^2Y).
\end{equation}
The torsion tensor satisfies the equality
\begin{equation}\label{tor1}
\begin{split}
T(AX,AY,Z)=-N(X,Y,Z)+dF(X,Y,AZ);\\
T(AX,Y,Z)=2(\nabla^g_XF)(Y,Z)-dF(X,Y,Z).
\end{split}
\end{equation}
The torsion connection $\nabla$ is determined by the formula
\begin{equation}\label{skct}
g(\nabla_XY,Z)=g(\nabla^g_XY,Z)+\frac12T(X,Y,Z).
\end{equation}
If in addition the skew-symmetric part $F$ of the generalized
Riemannian metric $G$ is closed, $dF=0$, then \eqref{ndf1} and
\eqref{tor1} hold inserting $dF=0$.
\end{thrm}
\begin{proof}
The equalities \eqref{tor1} follow by a straightforward
calculations from \eqref{st}, \eqref{nuj2}, \eqref{conn1} and
\eqref{conn2}. Now, the equality \eqref{ndf1} is an easy
consequence from the first equality in \eqref{tor1}. The formula
\eqref{skct} follows from \eqref{gencon} and \eqref{st}.
\end{proof}

\subsection{Eisenhart condition} Eisenhart \cite{Eis} was one of the first who proposed   a
connection with skew-symmetric torsion to be applied in general
relativity. In the sense of Eisenhart's definition (see \cite{Eis})
generalized Riemannian manifold  is a differentiable manifold with a
non-symmetric basic tensor $G(X,Y)=g(X,Y)+F(X,Y)$ with a connection
explictely defined by the equation
\begin{equation}\label{Eisen1}
g(\nabla_XY,Z)=\frac12\big[XG(Y,Z)+YG(Z,X)-ZG(Y,X)\Big].\end{equation}
It is easy to see that \eqref{Eisen1} can  be written in the form
\begin{equation*}
g(\nabla_XY,Z)=\frac12\big[XG(Y,Z)+YG(Z,X)-ZG(Y,X)\Big]=g(\nabla^g_XY,Z)+\frac12dF(X,Y,Z)\end{equation*}
which shows that the symmetric part $g$ of $G$ is covariantly
constant, $\nabla g=0$ and the torsion $T$  is totaly
skew-symmetric determined by the equation
$
T(X,Y,Z)=dF(X,Y,Z). $ 
Using relations (\ref{a2}) and
(\ref{nuj1}), we calculate that the Nijenhuis tensor satisfies
\begin{equation}\label{EisenhN1}
N(X,Y,Z)=(\nabla^g_{AX}F)(Y,Z)-(\nabla^g_{AY}F)(X,Z)+(\nabla^g_XF)(Y,AZ)-(\nabla^g_YF)(X,AZ).
\end{equation}
If, in addition, the Eisenhart connection preserves the
generalized Riemannian metric $G$, $\nabla g=\nabla F=0$, then
equation \eqref{EisenhN1} reduces to \eqref{nuj2} with $T=dF$.
\subsection{Skew-symmetric torsion. Examples} A generalized
Riemannian metric $G$ is equivalent to a choice of a
pseudo-Riemannian metric $g$ and a 2-form $F$ (an (1,1) tensor A
satisfying \eqref{m1}), such that $G=g+F$. A generalized metric
connection, i.e. a linear connection preserving $G$, is a
Riemannian connection preserving the 2-form $F$ or equivalently, a
Riemannian connection preserving the (1,1)-tensor $A$. This
supplies a number of examples.

\subsubsection{Almost Hermitian manifolds, $A^2=-1$}

Let us consider an almost Hermitian manifold $(M, g, A)$, i.e.
Riemannian manifold $(M,g)$ of dimension $n(=2m\geq4)$ endowed
with an almost complex structure,  the endomorphism $A$ satisfies
$
A^2=-I,\quad F(X,Y)=g(AX,Y), \quad g(AX,AY)=g(X,Y).
$ 
The 2-form $F$ is  the K\"ahler form (note the sign difference of
the K\"ahler form in \cite{FI}).

In this case the Nijenhuis tensor (\ref{nuj}) has the properties
\begin{equation}\label{propAHN1}
N(AX,Y,Z)=N(X,AY,Z)=N(X,Y,AZ)
\end{equation}
Then Theorem \ref{skew} gives the following well known result
\begin{cor}[\protect\cite{Str,Gau,FI}]\label{AHskew}
On an almost hermitian manifold $(M,g,A)$ there exists a unique
linear connection $\nabla$ preserving  the generalized Riemannian
metric $G=g+F$ with totally skew-symmetric torsion $T$ if and only
if the Nijenhuis tensor is totally skew-symmetric,
$N(X,Y,Z)=-N(X,Z,Y)$. The torsion T is determined by
$
T(X,Y,Z)=N(X,Y,Z)+dF(AX,AY,AZ). $ 
In particular,  an almost K\"ahler manifold, $dF=0$, admits such a
connection if and only if it is a K\"ahler manifold, $N=0$
\end{cor}

\subsubsection{Almost para-hermitian manifolds, $A^2=1$} Almost
para-Hermitian manifold $(M, g, A)$ is Riemannian manifold $(M,g)$
endowed with endomorphism $A$ satisfying
$
A^2=I,\quad F(X,Y)=g(AX,Y), \quad g(AX,AY)=-g(X,Y).
$ 
Such a manifold is  of even dimension $2n$, the
eigen-subbundles of the  paracomplex structure $A$ are of equal
dimension $n$ and the metric $g$ is of neutral signature (n,n). In
this case the Nijenhuis tensor (\ref{nuj}) has the properties
\eqref{propAHN1} and Theorem \ref{skew} gives the following well
known result (note the sign difference of the 2-form $F$ in
\cite{IZ})
\begin{cor}[\protect\cite{IZ}]\label{APHskew}
On an almost para-hermitian manifold $(M,g,A)$ there exists a
unique linear connection $\nabla$ preserving  the generalized
Riemannian metric $G=g+F$ with totally skew-symmetric torsion $T$
if and only if the Nijenhuis tensor is totally skew-symmetric,
$N(X,Y,Z)=-N(X,Z,Y)$. The torsion T is determined by
$
T(X,Y,Z)=-N(X,Y,Z)+dF(AX,AY,AZ). $ 
In particular, an almost para-K\"ahler manifold, $dF=0$, admits
such a connection if and only if it is a para-K\"ahler manifold,
$N=0$
\end{cor}

\subsubsection{Almost contact metric structures} We consider an
almost contact metric manifold $(M^{2n+1}, g, A, \eta,\xi)$, i.e.,
a $(2n+1)$ -dimensional Riemannian manifold equipped with a 1-form
$\eta$, a (1,1)-tensor $A$ and a vector field $\xi$ dual to $\eta$
with respect to the metric $g$, $\eta(\xi)=1, \eta(X)=g(X,\xi)$
such that the following compatibility conditions are satisfied
(see e.g. \cite{Blair})
\begin{equation}\label{acon}
A^2=-id+\eta\otimes\xi, \quad g(AX,AY)=g(X,Y)-\eta(X)\eta(Y), \quad
F(X,Y)=g(AX,Y), \quad A\xi=0.
\end{equation}
In this case the skew-symmetric part $F$ of $G=g+F$ is degenerate,
$F(\xi,X)=0$ and has rank $2n$. Then Theorem~\ref{main} together
with \eqref{acon} implies $
0=g((\nabla_XA)AY,\xi)=(\nabla_X\eta)Y=g(\nabla_X\xi,Y), \quad
i.e. \quad \nabla\eta=\nabla\xi=0. $ 
 Now
\eqref{skct} yields $0=g(\nabla^g_X\xi,Y)+\frac12T(X,\xi,Y)$ which
shows that $\xi$ is a Killing vector field because $T$ is
skew-symmetric and
\begin{equation}\label{aconeta}
d\eta=\xi\lrcorner T,\qquad \xi\lrcorner d\eta=0,
\end{equation}
where $\lrcorner$ is the interior multiplication.

We obtain from \eqref{nuj2} applying \eqref{acon} and
\eqref{aconeta} that
\begin{equation}\label{nujac}
N(X,Y,Z)+\eta(Z)d\eta(X,Y)=-T(AX,AY,Z)+T(X,Y,Z)-T(AX,Y,AZ)-T(X,AY,AZ)
\end{equation}
The right hand side of equation \eqref{nujac} is totally
skew-symmetric which yields that the tensor
\begin{equation}\label{nujac0}N^{ac}=N+d\eta\otimes \eta
\end{equation}
 is totally skew-symmetric, $N^{ac}(X,Y,Z)=-N^{ac}(X,Z,Y)$.
 In fact, the tensor $N^{ac}$ is usually called \emph{the Nijenhuis tensor in almost contact geometry} (see e.g. \cite{Blair}).

We get from \eqref{nujac} applying \eqref{acon} and \eqref{aconeta}
that
\begin{equation}\label{nujac1}
N(\xi,Y,Z)=d\eta(Y,Z)-d\eta(AY,AZ).
\end{equation}
We calculate from  \eqref{nuj} taking into account \eqref{acon}
and \eqref{aconeta} that
\begin{multline}\label{acs}
N(AX,AY,Z)=-N(X,Y,Z)\\-\eta(Z)\Big[d\eta(AX,AY)+d\eta(X,Y)\Big]+\eta(X)N(\xi,Y,Z)-\eta(Y)N(\xi,X,Z)
\end{multline}
The first formula in \eqref{tor1} together with \eqref{aconeta} and
\eqref{acs} yields
\begin{multline*}
T(X,Y,Z)=-N(AX,AY,Z)+dF(AX,AY,AZ)+\eta(X)d\eta(Y,Z)+\eta(Y)d\eta(Z,X)\\=N(X,Y,Z)+dF(AX,AY,AZ)
+\eta(Z)\Big[d\eta(AX,AY)+d\eta(X,Y)\Big]-\eta(X)N(\xi,Y,Z)+\eta(Y)N(\xi,X,Z)\\
+\eta(X)d\eta(Y,Z)+\eta(Y)d\eta(Z,X).
\end{multline*}
 We obtain from this equality applying \eqref{nujac1} the following formula for the skew-symmetric torsion
\begin{multline}\label{tac}
T(X,Y,Z)=N(X,Y,Z)+\eta(Z)d\eta(X,Y)+dF(AX,AY,AZ)\\+\eta(Z)d\eta(AX,AY)+\eta(Y)d\eta(AZ,AX)
+\eta(X)d\eta(AY,AZ).
\end{multline}
Observe using \eqref{nuj}, \eqref{acon} and \eqref{nujac0} that
$d\eta(AX,AY)=-N(X,Y,\xi)=-N^{ac}(X,Y,\xi)+d\eta(X,Y).$ Substitute
the last equality into \eqref{tac} and use the skew-symmetricity
of $N^{ac}$ to get the
 equality obtained in \cite{FI}
\begin{equation}\label{tac1}
T(X,Y,Z)=\eta\wedge d\eta+N^{ac}-d^AF-\eta(Z)\wedge(\xi\lrcorner
N^{ac}), \quad d^AF(X,Y,Z)=-dF(AX,AY,AZ).
\end{equation}
Theorem~\ref{skew} gives the next well known result (note the sign
difference of the 2-form $F$ in \cite{FI})
\begin{cor}[\protect\cite{FI}]\label{AcHskew}
On an almost contact metric manifold manifold $(M,g,A,F,\eta,\xi)$
there exists a unique linear connection $\nabla$ preserving  the
generalized Riemannian metric $G=g+F$ with totally skew-symmetric
torsion $T$ if and only if the almost contact Nijenhuis tensor
$N^{ac}$ is totally skew-symmetric,  $N^{ac}(X,Y,Z)=-N^{ac}(X,Z,Y)$
and $\xi$ is a Killing vector field. The torsion T is determined by
\eqref{tac1} or, equivalently by \eqref{tac}.
\end{cor}

\subsubsection{Almost paracontact metric structures} We consider an
almost paracontact metric manifold $(M^{2n+1}, g, A, \eta,\xi)$,
i.e., a $(2n+1)$ -dimensional pseudo Riemannian manifold of
signature (n+1,n) equipped with a 1-form $\eta$, a (1,1)-tensor
$A$ and a vector field $\xi$ dual to $\eta$ with respect to the
metric $g$, $\eta(\xi)=1, \eta(X)=g(X,\xi)$ such that the
following compatibility conditions hold (see e.g. \cite{Zam})
\begin{equation}\label{apcon}
A^2=id-\eta\otimes\xi, \quad g(AX,AY)=-g(X,Y)+\eta(X)\eta(Y), \quad
F(X,Y)=g(AX,Y), \quad A\xi=0.
\end{equation}
In this case the skew-symmetric part $F$ of $G=g+F$ is degenerate,
$F(\xi,X)=0$ and has rank $2n$. Theorem~\ref{main} together with
\eqref{apcon} implies $
0=g((\nabla_XA)AY,\xi)=-(\nabla_X\eta)Y=-g(\nabla_X\xi,Y), \quad
i.e. \quad \nabla\eta=\nabla\xi=0. $ 
 Now \eqref{skct} yields $0=g(\nabla^g_X\xi,Y)+\frac12T(X,\xi,Y)$
which shows that $\xi$ is a Killing vector field because $T$ is
skew-symmetric and $
d\eta=\xi\lrcorner T,\qquad \xi\lrcorner d\eta=0. $ 
We obtain from \eqref{nuj2} applying \eqref{apcon} and the last
equality 
\begin{equation}\label{nujapc}
N(X,Y,Z)-\eta(Z)d\eta(X,Y)=-T(AX,AY,Z)-T(X,Y,Z)-T(AX,Y,AZ)-T(X,AY,AZ)
\end{equation}
The right hand side of equation \eqref{nujapc} is totally
skew-symmetric which yields that the tensor
$
N^{apc}=N-d\eta\otimes \eta $ 
 is totally skew-symmetric, $N^{apc}(X,Y,Z)=-N^{apc}(X,Z,Y)$.
 In fact, the tensor $N^{apc}$ is usually called \emph{the Nijenhuis tensor in almost paracontact geometry} (see e.g. \cite{Zam}).

As in the previous case,  Theorem~\ref{skew} gives the following
well known result (note the sign difference of the 2-form $F$ in
\cite{Zam})
\begin{cor}[\protect\cite{Zam}]\label{AcHskewp}
On an almost paracontact metric manifold manifold
$(M,g,A,F,\eta,\xi)$ there exists a unique linear connection
$\nabla$ preserving  the generalized Riemannian metric $G=g+F$
with totally skew-symmetric torsion $T$ if and only if the almost
paracontact Nijenhuis tensor $N^{apc}$ is totally skew-symmetric,
$N^{apc}(X,Y,Z)=-N^{apc}(X,Z,Y)$ and $\xi$ is a Killing vector
field. The torsion T is determined by
\begin{equation*}
T(X,Y,Z)=\eta\wedge d\eta-N^{apc}-d^AF+\eta(Z)\wedge(\xi\lrcorner
N^{ac}), \quad d^AF(X,Y,Z)=-dF(AX,AY,AZ).
\end{equation*}
\end{cor}


\section{Einstein metricity condition (NGT)}\label{ngt}
In his attempt to construct an unified field theory (Non-symmetric
Gravitational Theory, briefly NGT) A.Einstein \cite{Ein}
considered a generalized Riemannian manifold and use the so called
metricity condition \eqref{metein} which can be written as follows
$
XG(Y,Z)-G(\nabla_YX,Z)-G(Y,\nabla_XZ)=0, $ 
see also \cite{Mof,LMof} for a little bit more general right hand
side of \eqref{metein}.

In view of \eqref{covG},  the definition of the torsion,
\eqref{metric} and \eqref{m1}  the metricity condition
\eqref{metein} can be written in the form
\begin{equation}\label{metein1}
(\nabla_XG)(Y,Z)=-G(T(X,Y),Z) \quad \Leftrightarrow \quad
(\nabla_X(g+F))(Y,Z)=-T(X,Y,Z)+T(X,Y,AZ).
\end{equation}
A general solution for the connection $\nabla$ satisfying
\eqref{metein} is given in  terms of   $g,F,T$\cite{Hlav} (see also
\cite{Mof}). Here we show that a solution can be expressed in terms
of the exterior derivative $dF$ of $F$ and find formulas for the
covariant derivatives of $\nabla g$ and $\nabla F$.

Taking the cyclic sum in \eqref{metein1} and applying \eqref{cconn1}
, we obtain
\begin{multline}\label{ein2}
(\nabla_Xg)(Y,Z)+(\nabla_Yg)(Z,X)+(\nabla_Zg)(X,Y)\\=-dF(X,Y,Z)-T(X,Y,Z)-T(Y,Z,X)-T(Z,X,Y).
\end{multline}
It is easy to observe that the left hand side of \eqref{ein2} is
symmetric while the right hand side  is skew-symmetric. Hence, we
get
\begin{equation}\label{ein3}
\begin{split}
(\nabla_Xg)(Y,Z)+(\nabla_Yg)(Z,X)+(\nabla_Zg)(X,Y)=0;\\
dF(X,Y,Z)=-T(X,Y,Z)-T(Y,Z,X)-T(Z,X,Y).
\end{split}
\end{equation}
The symmetric part of \eqref{metein1} with respect to $X,Y$ gives
$$(\nabla_Xg)(Y,Z)+(\nabla_Yg)(X,Z)+(\nabla_XF)(Y,Z)+(\nabla_YF)(X,Z)=0,$$
which combined with the first equality in \eqref{ein3} yields
\begin{equation}\label{ein4}
(\nabla_Zg)(X,Y)=(\nabla_XF)(Y,Z)+(\nabla_YF)(X,Z).
\end{equation}
Substitute \eqref{ein4} into \eqref{metein1} and use \eqref{cconn1}
to get
\begin{equation}\label{ein5}
(\nabla_ZF)(X,Y)=\frac12\Big[dF(X,Y,Z)+T(X,Y,Z)-T(Z,Y,AX)+T(Z,X,AY)\Big].
\end{equation}
We obtain inserting \eqref{ein5} into \eqref{ein4} that
\begin{equation}\label{ein6}
(\nabla_Xg)(Y,Z)=-\frac12\Big[T(X,Y,Z)-T(X,Y,AZ)+T(X,Z,Y)-T(X,Z,AY)\Big].
\end{equation}
Applying  \eqref{ein6} and the second equation in \eqref{ein3} we
obtain from \eqref{cgencon} that
\begin{multline}\label{genconein}
g(\nabla_XY,Z)=g(\nabla^g_XY,Z)+\frac12\Big[T(X,Y,Z)-T(X,Z,AY)-T(Y,Z,AX)\Big]\\
=g(\nabla^g_XY,Z)-\frac12\Big[dF(X,Y,Z)+T(Z,X,Y)+T(Y,Z,X)\Big]+
\frac12\Big[T(Z,X,AY)+T(Z,Y,AX)\Big].
\end{multline}
\subsection{NGT involving the Nijenhuis tensor}
We are going to involve the Nijenhuis tensor in our attemp to
determine the torsion $T$. We start with an expression of the
covariant derivative $\nabla A$. Substitute \eqref{ein5} and
\eqref{ein6} into \eqref{a2}, we find
\begin{multline}\label{ein7}
g((\nabla_XA)Y,Z)=(\nabla_XF)(Y,Z)-(\nabla_Xg)(AY,Z)\\=
\frac12\Big[dF(X,Y,Z)+T(Y,Z,X)+T(X,Y,AZ)+T(X,AY,Z)-T(X,AY,AZ)-T(X,Z,A^2Y)\Big]
\end{multline}

Substitute \eqref{ein7} into \eqref{nuj1} to get after some
calculations the following equality

\begin{multline}\label{ein8}
N(X,Y,Z)=dF(X,Y,AZ)+\frac12\big[dF(AX,Y,Z)+dF(X,AY,Z)\big]\\+
\frac12\Big[T(Y,Z,AX)-T(X,Z,AY)+T(Y,AZ,X)-T(X,AZ,Y)-T(X,AY,A^2Z)-T(AX,Y,A^2Z)\Big]\\
-\frac12\Big[T(Z,AY,A^2X)-T(Z,AX,A^2Y)+T(AZ,Y,A^2X)-T(AZ,X,A^2Y)\Big]-T(AX,AY,AZ).
\end{multline}

\subsection{NGT and skew-symmetric torsion}

Now we consider the case of totally skew-symmetric torsion,
$T(X,Y,Z)=-T(X,Z,Y)$. From the considerations above, we have

\begin{thrm}\label{pp1}
A generalized Riemannian manifold $(M,G=g+F)$ admits a linear
connection satisfying the Einstein metricity condition
\eqref{metein} with totally skew-symmetric torsion $T$ if and only
if the Nijenhuis tensor $N$ and the exterior derivative of $F$
satisfy the relation
\begin{multline}\label{skew1}
N(X,Y,Z)=\frac23dF(X,Y,AZ)+\frac13dF(AX,Y,Z)+\frac13dF(X,AY,Z)+\frac13dF(AX,AY,AZ)\\
-\frac16\Big[dF(A^2X,Y,AZ)+dF(A^2X,AY,Z)+dF(X,A^2Y,AZ)-dF(X,AY,A^2Z)\Big]\\-\frac16\Big[dF(AX,A^2Y,Z)-dF(AX,Y,A^2Z)\Big]
\end{multline}
In this case the skew-symmetric torsion is completely determined
by the exterior derivative of the skew-symmetric part of the
generalized Riemannian metric,
\begin{equation}\label{tordfnew}
T(X,Y,Z)=-\frac13dF(X,Y,Z), \end{equation} the Einstein metricity
condition has the form
$(\nabla_XG)(Y,Z)=\frac13\Big[dF(X,Y,Z)-dF(X,Y,AZ)\Big]$ and it is
equivalent to the following two conditions
\begin{equation}\label{skew0}
\begin{split}
(\nabla_Xg)(Y,Z)=-\frac16\Big[dF(X,Y,AZ)-dF(X,AY,Z)\Big];\\
(\nabla_XF)(Y,Z)=\frac16\Big[2dF(X,Y,Z)-dF(X,Y,AZ)-dF(X,AY,Z)\Big].
\end{split}
\end{equation}
The connection is unique determined by the formula
\begin{equation}\label{newnbl}
g(\nabla_XY,Z)=g(\nabla^g_XY,Z)-\frac16dF(X,Y,Z)-\frac16dF(X,AY,Z)+\frac16dF(AX,Y,Z).
\end{equation}
The covariant derivative of $F$ and $A$ with respect to the
Levi-Civita connection $\nabla^g$ are given by
 \begin{multline}\label{ff2}
(\nabla^g_XF)(Y,Z) =g((\nabla^g_XA)Y,Z)\\= \frac{1}{3}
dF(X,Y,Z)+\frac{1}{3}
dF(X,AY,AZ)-\frac16dF(AX,Y,AZ)-\frac16dF(AX,AY,Z).
\end{multline}
\end{thrm}
\begin{proof}
The second equation of \eqref{ein3} together with the
skew-symmetric property ot the torsion imply \eqref{tordfnew}.
Using the already obtained \eqref{tordfnew} reduces \eqref{ein8}
to \eqref{skew1}. Now, \eqref{skew0} and \eqref{newnbl} follow
from \eqref{ein5}, \eqref{ein6}, \eqref{genconein}  and
\eqref{tordfnew}. Applying \eqref{tordfnew} and \eqref{skew0} to
\eqref{cconn2}, we obtain
\begin{multline}\label{ff1}
(\nabla_XF)(Y,Z)=(\nabla^g_XF)(Y,Z)-\frac16\Big[dF(X,Y,AZ)+dF(Z,X,AY)\Big]\\-
\frac16\Big[2dF(X,AY,AZ)+dF(Z,AY,AX) +dF(Y,AX,AZ)\Big].
\end{multline}
A substitution of the second formula in (\ref{skew0}) into
(\ref{ff1}) gives \eqref{ff2}.

For the converse, it is straightforward to check that the
connection determining by \eqref{newnbl} satisfies the required
properties provided \eqref{skew1} holds. This completes the proof.
\end{proof}
%
%

\subsection{Almost Hermitian manifolds, $A^2=-1$} The condition \eqref{skew1} in the almost hermitian case takes the form
\begin{equation}\label{skewah}
N(X,Y,Z)=dF(X,Y,AZ)+\frac13dF(AX,Y,Z)+\frac13dF(X,AY,Z)+\frac13dF(AX,AY,AZ).
\end{equation}
Using the properties of the Nijenhuis tensor of an almost Hermitian
manifold, \eqref{propAHN1}, we get from \eqref{skewah} that the
exterior derivative $dF$ of the K\"ahler form satisfies
\begin{equation}\label{dfah}
dF(X,Y,AZ)=dF(X,AY,Z)=dF(AX,Y,Z)=-dF(AX,AY,AZ),
\end{equation}
i.e. $dF$ is of type $(3,0)+(0,3)$ with respect to the almost
complex structure $A$. Substitute \eqref{dfah} into \eqref{skewah}
to get
$
N(X,Y,Z)=\frac43dF(X,Y,AZ).
$

At this point we recall that an almost Hermitian manifold is said
to be Nearly K\"ahler if the covariant derivative of the almost
complex structure $A$ with respect to the Levi-Civita connection
$\nabla^g$ of the metric $g$ is skew-symmetric,
\begin{equation}\label{nk2}
(\nabla^g_XA)X=0 \quad \Leftrightarrow \quad (\nabla^g_XF)(X,Y)=0.
\end{equation}
The  Nearly K\"ahler manifolds (called almost Tachibana spaces in
\cite{Yano}) were developed by A. Gray \cite{Gr1,Gr2,Gr3} and
intensively studied since then in \cite{Ki,N1,MNS,N2,But,FIMU}.
Nearly K\"ahler manifolds in dimension 6 are Einstein manifolds
with positive scalar curvature, the Nijenhuis tensor $N$   is a
3-form and it is parallel with respect to the Gray characteristic
connection (see \cite{Ki}). This connection was defined by Gray
\cite{Gr1,Gr2,Gr3} and it turns out to be the unique linear
connection preserving the nearly K\"ahler structure and having
totally skew-symmetric torsion (see \cite{FI}). Nearly K\"ahler
manifolds appear also in supersymmetric string theories (see e.g.
\cite{Po1,PS,LNP,GLNP} etc.

We obtain from Theorem~\ref{pp1} that

\begin{thrm}\label{nika}
Let $(M,A,g,F)$ be an almost Hermitian manifold with a K\"ahler
2-form F considered as a  generalized Riemannian manifold
$(M,G=g+F)$. Then $(M,G)$ satisfies the Einstein metricity
condition \eqref{metein} with a totally skew-symmetric torsion $T$
if and only if it is a Nearly K\"ahler manifold.

The skew-symmetric torsion is determined by the condition
\begin{equation}\label{nk3}
T(X,Y,Z)=-\frac13dF(X,Y,Z)=\frac14N(X,Y,AZ).
\end{equation}
The connection is unique given by the formula
$$
g(\nabla_XY,Z)=g(\nabla^g_XY,Z)-\frac16dF(X,Y,Z)=g(\nabla^g_XY,Z)+\frac18N(X,Y,AZ).
$$
The Einstein metricity condition has the form
$(\nabla_XG)(Y,Z)=\frac13\Big[dF(X,Y,Z)-dF(X,Y,AZ)\Big].$ The
covariant derivative of $g$ and the K\"ahler form $F$ are
\begin{equation*}
\begin{split}
(\nabla_Xg)(Y,Z)=0; \qquad
(\nabla_XF)(Y,Z)=\frac13\Big[dF(X,Y,Z)-dF(X,Y,AZ)\Big].
\end{split}
\end{equation*}
\end{thrm}
\begin{proof}
Applying  \eqref{nk3} and \eqref{dfah} to \eqref{cconn2} we
conclude  that the Nearly K\"ahler condition \eqref{nk2} holds.
The rest of the theorem follow from Theorem~\ref{pp1}.
\end{proof}

\subsection{Almost para-Hermitian manifolds, $A^2=1$} The condition \eqref{skew1} in the almost para-hermitian case takes the form
\begin{equation*}
N(X,Y,Z)=\frac13dF(X,Y,AZ)+\frac13dF(AX,Y,Z)+\frac13dF(X,AY,Z)+\frac13dF(AX,AY,AZ),
\end{equation*}
which implies that the Nijenhuis tensor is totally skew symmetric.

We obtain from Theorem~\ref{pp1} and \cite[Proposition~3.1]{IZ}
\begin{thrm}\label{ngtskewp}
Let $(M,A,g,F)$ be an almost para-Hermitian manifold with a
K\"ahler 2-form F considered as a  generalized Riemannian manifold
$(M,G=g+F)$. Then $(M,G)$ satisfies the Einstein metricity
condition \eqref{metein} with a totally skew-symmetric torsion $T$
if and only if the Nijenhuis tensor is totally skew-symmetric.

The skew-symmetric torsion is determined by
$T(X,Y,Z)=-\frac13dF(X,Y,Z)$, the Einstein  metricity condition
has the form
$(\nabla_XG)(Y,Z)=\frac13\Big[dF(X,Y,Z)-dF(X,Y,AZ)\Big]$ and the
equalities \eqref{skew0}, \eqref{newnbl} and \eqref{ff2} hold with
$A^2=1$.

\end{thrm}
\begin{rmrk}
In particular, the nearly paraK\"ahler manifolds defined by
$(\nabla^g_XA)X=0$ have totally skew symmetric Nijenhuis tensor,
$dF$ share the properties \eqref{propAHN1} of the Nijenhuis tensor
and $N(X,Y,Z)=\frac43dF(X,Y,AZ)$ \cite[Proposition~5.1]{IZ}. In
this case the  NGT connection with skew symmetric torsion
preserves the symmetric part of $G$, $\nabla g=0$.
\end{rmrk}
\begin{rmrk}
Combining Theorem~\ref{ngtskewp}  and Corollary~\ref{APHskew} we
conclude that an almost para-Hermitian manifold admits an NGT
connection with totally skew symmetric torsion if and only if it
admits a linear connection preserving the almost para-Hermitian
structure with a totally skew-symmetric torsion.
\end{rmrk}

\subsection{Almost contact metric structures.}
Let us consider the case of an almost contact metric structure
which is defined in the Section 4.3. The fact $g(\xi,\xi)=1$ and
\eqref{ff2} imply
\begin{equation}\label{kill}
(\nabla^g_X\eta)Z=g(\nabla^g_X\xi,Z)=-g((\nabla^g_XA)\xi,AZ)=\frac13dF(X,AZ,\xi)+\frac16dF(AX,Z,\xi).
\end{equation}
Consequently, we obtain
\begin{equation}\label{xi1}
d\eta(X,Z)=\frac12dF(X,AZ,\xi)+\frac12dF(AX,Z,\xi), \quad
d\eta(X,\xi)=0, \quad d\eta(AX,Z)=d\eta(X,AZ).
\end{equation}
Applying \eqref{acon} and \eqref{xi1}, we simplify \eqref{skew1}
to get
\begin{multline}\label{skewac}
N(X,Y,Z)=dF(X,Y,AZ)+\frac13dF(AX,Y,Z)+\frac13dF(X,AY,Z)+\frac13dF(AX,AY,AZ)\\
-\frac13\Big[d\eta(Y,Z)\eta(X)+d\eta(Z,X)\eta(Y)-d\eta(X,Y)\eta(Z)\Big].
\end{multline}

On the other hand, \cite[Lemma~6.1]{Blair}, in our notation, reads

\begin{equation}\label{Problem}
\aligned 2g((\nabla^g_XA)Y,Z)&=dF(X,Y,Z)-dF(X,AY,AZ)+N^{ac}(Y,Z,AX)\\
&+\Big[d\eta(AY,Z)-d\eta(AZ,Y)\Big]\eta(X)-d\eta(X,AY)\eta(Z)+d\eta(X,AZ)\eta(Y).
\endaligned
\end{equation}
Substitute \eqref{ff2} into \eqref{Problem}, use \eqref{xi1} and
observe $N^{ac}(Y,Z,AX)=N(Y,Z,AX)$ to get
\begin{multline}\label{skewacB}
N(Y,Z,AX)=-2d\eta(AY,Z)\eta(X)+d\eta(X,AY)\eta(Z)-d\eta(X,AZ)\eta(Y)\\-\frac13dF(X,Y,Z)+\frac53dF(X,AY,AZ)-\frac13dF(AX,Y,AZ)-\frac13dF(AX,AY,Z).
\end{multline}
Setting $X=\xi$ into \eqref{skewacB} gives
$0=-\frac13dF(\xi,Y,Z)+\frac53dF(\xi,AY,AZ)-2d\eta(AY,Z) $ which,
in view of  \eqref{xi1}, yields
\begin{equation}\label{dfdet}
dF(Y,AZ,\xi)=d\eta(Y,Z)=dF(AY,Z,\xi).
\end{equation}
Applying \eqref{dfdet} to \eqref{kill} leads to
\begin{equation}\label{kill1}
(\nabla^g_X\eta)Y=g(\nabla^g_X\xi,Y)=\frac12d\eta(X,Y),
\end{equation}
which shows that the vector field $\xi$  is a Killing vector
field.

The equalities \eqref{skewac} and \eqref{dfdet} imply
\begin{equation}\label{nijxi}
N(X,Y,\xi)=N(\xi,X,Y)=d\eta(X,Y).
\end{equation}
Applying \eqref{nijxi}, we obtain from \eqref{skewacB}
\begin{multline}\label{skewacB1}
N(X,Y,Z)=d\eta(Y,Z)\eta(X)+\frac53d\eta(X,Y)\eta(Z)+d\eta(Z,X)\eta(Y)\\+\frac13dF(X,Y,AZ)-\frac53dF(AX,AY,AZ)-\frac13dF(AX,Y,Z)-\frac13dF(X,AY,Z).
\end{multline}
Compare \eqref{skewac} with \eqref{skewacB1} to  derive
\begin{equation}\label{finac}
3dF(AX,AY,AZ)+dF(X,Y,AZ)+dF(AX,Y,Z)+dF(X,AY,Z)=2(d\eta\wedge\eta)(X,Y,Z)
\end{equation}
Using \eqref{acon} and \eqref{dfdet}, we derive from \eqref{finac}
the following
\begin{equation}\label{finac1}
dF(AX,AY,AZ)+dF(X,Y,AZ)=\eta(X)d\eta(Y,Z)+\eta(Y)d\eta(Z,X).
\end{equation}
Substitute \eqref{finac} and \eqref{finac1} into \eqref{skewacB1}
to get
\begin{equation}\label{skewacB2}
N(X,Y,Z)=-\frac43dF(AX,AY,AZ) + (d\eta\wedge\eta)(X,Y,Z),
\end{equation}
which, in particular, shows that the Nijenhuis tensor $N$ is
totally skew symmetric.

Applying \eqref{finac1} to \eqref{ff2} we obtain
 \begin{multline}\label{ff2f}
g((\nabla^g_XA)Y,Z)= \frac{1}{3}
dF(X,Y,Z)-\frac16\eta(Y)d\eta(Z,AX)+\frac13\eta(X)d\eta(Y,AZ)
+\frac16\eta(Z)d\eta(Y,AX)\\=-\frac13dF(AX,AY,Z)+\frac16\eta(Z)d\eta(Y,AX)-\frac12\eta(Y)d\eta(AZ,X).
\end{multline}
We derive from \eqref{ff2f} that $
g((\nabla^g_XA)Y,Z)+g((\nabla^g_YA)X,Z)=-\frac12\eta(Y)d\eta(AZ,X)-\frac12\eta(X)d\eta(AZ,Y),
$
i.e. the structure is nearly-cosymplectic, $(\nabla^g_XA)X=0$ if
and only if the 1-form $\eta$ is closed, $d\eta=0$.

It seems reasonable to consider the following
\begin{dfn}
An almost contact metric manifold $(M^{2n+1},A,g,F,\eta,\xi)$ is
said to be \emph{almost-nearly cosymplectic} if the Levi-Civita
covariant derivative of the fundamental 2-form satisfies the
following condition
\begin{equation}\label{ff3f}
g((\nabla^g_XA)Y,Z)=-\frac13dF(AX,AY,Z)+\frac16\eta(Z)d\eta(Y,AX)-\frac12\eta(Y)d\eta(AZ,X).
\end{equation}
In particular, the vector field $\xi$ is Killing and
\eqref{skewacB2} holds.
\end{dfn}

Theorem~\ref{pp1} yields

\begin{thrm}\label{acnika}
Let $(M,A,g,F,\eta,\xi)$ be an almost contact metric manifold with
a fundamental 2-form F considered as a generalized Riemannian
manifold $(M,G)$ with a generalized Riemannian metric $G=g+F$.
Then $(M,G)$ satisfies the Einstein metricity condition
\eqref{metein} with a totally skew-symmetric torsion $T$ if and
only if it is an almost-nearly cosymplectic, i.e. \eqref{ff3f}
holds.

The skew-symmetric torsion is determined by the condition
\begin{multline*}
T(X,Y,Z)=-\frac13dF(X,Y,Z)\\=-\frac14N(AX,AY,AZ)+\frac13\Big[
\eta(X)d\eta(Y,AZ)+\eta(Y)d\eta(Z,AX)+\eta(Z)d\eta(X,AY)\Big].
\end{multline*}
The connection is unique given by the formula
$$
g(\nabla_XY,Z)=g(\nabla^g_XY,Z)-\frac16dF(X,Y,Z)+\frac16\Big[\eta(X)d\eta(Y,Z)+\eta(Y)d\eta(X,Z)
\Big].
$$
The Einstein metricity condition has the form
$(\nabla_XG)(Y,Z)=\frac13\Big[dF(X,Y,Z)-dF(X,Y,AZ)\Big]$. The
covariant derivative of  $g$ and $F$ are
$(\nabla_Xg)(Y,Z)=\frac16\Big[\eta(Y)d\eta(Z,X)+\eta(Z)d\eta(Y,X)\Big]$,
\begin{equation*}
\begin{split}
 (\nabla_XF)(Y,Z)=\frac13\Big[dF(X,Y,Z)-dF(X,Y,AZ)\Big]-\frac16\Big[\eta(Y)d\eta(Z,X)+\eta(Z)d\eta(Y,X)\Big].
\end{split}
\end{equation*}
\end{thrm}
\subsection{Almost-nearly cosymplectic structures} We list some
elementary properties of the almost-nearly cosimplectic structures
defined by \eqref{ff3f}.

First, we have due to \eqref{kill1}
\begin{cor}On an almost-nearly cosimplectic manifold the vector field $\xi$
is  Killing.
\end{cor}
\begin{cor}
If an almost-nearly cosymplectic structure  is   normal  then it
is cosymmplectic (coK\"ahler), i.e. $d\eta=dF=0$.
\end{cor}
\begin{proof}
An almost contact structure is normal if the almost contact
Nijenhuis tensor vanishes, $N^{ac}=0$ which, in view of
\eqref{nujac0}, gives $N(X,Y,Z)=-\eta(Z)d\eta(X,Y)$. Now,
\eqref{skewacB2} yields $d\eta=0$ and $dF(AX,AY,AZ)=0$. The last
equality together with \eqref{dfdet} implies $dF=0$.
\end{proof}
\begin{rmrk}
An almost-nearly cosymplectic structure is never contact because
of \eqref{xi1}. Indeed if $d\eta=F$ then $d\eta$ is of type (1,1)
with respect to $A$ which contradicts \eqref{xi1}.
\end{rmrk}
Suppose an almost-nearly cosymplectic structure has closed 1-form
$\eta, d\eta=0$. Then \eqref{xi1} yields that $\eta$ is
$\nabla^g$-parallel, $\nabla^g\eta=0$. Then the distribution
$H=Ker(\eta)$ is involutive and therefore it is Frobenius
integrable. The integral submanifold $N^{2n}$ is a nearly K\"ahler
manifold due to \eqref{ff3f}. Hence, an almost-nearly cosymplectic
manifold with a closed 1-form $\eta$ is locally a product of a
nearly K\"ahler manifold with the real line.

More general, we have  $\mathcal L_{\xi}\eta=d\eta(\xi,.)=0$ due
to \eqref{xi1}. Since $\xi$ is a Killing vector field, we
calculate using \eqref{kill1}, \eqref{dfdet} and \eqref{ff2f} that
\begin{equation}\label{lieF}
\mathcal
(L_{\xi}F)(X,Y)=(\nabla^g_{\xi}F)(X,Y)-g(\nabla^g_X\xi,AY)+g(\nabla^g_Y\xi,AX)=d\eta(Y,AX).
\end{equation}
Suppose  $M^{2n+1}$ is compact and the almost contact metric
structure is regular, i.e. the 2n-dimensional quotient space
$M^{2n+1}/\xi$ is a smooth manifold. If in addition  $d\eta=0$
then \eqref{lieF} and \eqref{ff2f} imply that $M^{2n+1}$ is a
trivial circle bundle over the Nearly K\"ahler manifold
$M^{2n+1}/\xi$. Conversely, start with a Nearly K\"ahler manifold
$N^{2n}$ we obtain on $M^{2n+1}\times\mathbb R^+$ an almost-nearly
cosymplectic structure  taking $\eta=dt$.

\subsection{Almost paracontact metric structures}  Almost paracontact
metric structure is defined in  Section 4.6. Equation \eqref{ff2}
 together with equation \eqref{apcon} implies the equalities
\begin{equation}\label{xi1AP}
\aligned
(\nabla^g_X\eta)Y=g(\nabla^g_X\xi,Y)=-g((\nabla^g_XA)AY,\xi)=-\frac{1}{3}dF(X,AY,\xi)+\frac16dF(AX,Y,\xi);\\
d\eta(X,Y)=
-\frac16dF(AX,Y,\xi)-\frac16dF(X,AY,\xi), \\ d\eta(X,\xi)=0, \quad
d\eta(AX,Y)=d\eta(X,AY).
\endaligned \end{equation}
With the help of \eqref{xi1AP} we write the condition
(\ref{skew1}) in the form
\begin{multline}\label{skew1AP}
N(X,Y,Z)=\frac13dF(X,Y,AZ)+\frac13dF(AX,Y,Z)+\frac13dF(X,AY,Z)+\frac13dF(AX,AY,AZ)\\
-\eta(X)d\eta(Y,Z)-\eta(Y)d\eta(Z,X)+\eta(Z)d\eta(X,Y)
\end{multline}


Now Theorem~\ref{pp1} leads to the following
\begin{thrm}\label{acnikap}
Let $(M,A,g,F,\eta,\xi)$ be an almost paracontact metric manifold
with a fundamental 2-form F considered as a generalized Riemannian
manifold $(M,G=g+F)$. Then $(M,G)$ satisfies the Einstein
metricity condition \eqref{metein} with a totally skew-symmetric
torsion $T$ if and only if the condition \eqref{ff2} holds

The skew-symmetric torsion is determined by \eqref{tordfnew},  the
connection is unique given by the formula \eqref{newnbl}.  The
covariant derivatives of $g$ and $F$ are determined  by
\eqref{skew0}.
\end{thrm}
\begin{proof}
We  show that \eqref{ff2} implies \eqref{skew1AP} and then apply
Theorem~\ref{pp1}.

The general formula from \cite[Proposition~2.4]{Zam},  in our
notation, reads
\begin{equation}\label{ZamAP}
\aligned 2g((\nabla^g_XA)Y,Z)&=dF(X,Y,Z)+dF(X,AY,AZ)-N^{apc}(Y,Z,AX)\\
&+\Big[d\eta(AY,Z)+d\eta(Y,AZ)\Big]\eta(X)+d\eta(AY,X)\eta(Z)-d\eta(AZ,X)\eta(Y).
\endaligned
\end{equation}

Substitute (\ref{ff2}) into (\ref{ZamAP}) and notice
$N^{apc}(Y,Z,AX)=N(Y,Z,AX)$, we get
\begin{multline}\label{ZamkAP}
 N(Y,Z,AX)=\frac13 dF(X,Y,Z)+\frac 13dF(X,AY,AZ)+\frac13dF(AX,Y,AZ)+\frac13dF(AX,AY,Z)\\
+2d\eta(AY,Z)\eta(X)+d\eta(AY,X)\eta(Z)-d\eta(AZ,X)\eta(Y).
\end{multline}
The defining equation \eqref{nuj} of the Nijenhuis tensor together
with the properties listed in \eqref{apcon} and \eqref{xi1AP}
imply $N(X,Y,\xi)=-d\eta(AX,AY)=-d\eta(X,Y)$. Now, we easily get
the equivalences of \eqref{ZamkAP} and (\ref{skew1AP}). An
application of  Theorem~\ref{pp1} completes the proof.
\end{proof}
In view of Theorem~\ref{acnikap} it seems reasonable to give the
following
\begin{dfn}
An almost paracontact metric manifold $(M,g,\eta,\xi,A)$ is called
\emph{an almost paracontact NGT manifold with torsion} if the
following condition holds
\begin{multline*}
g((\nabla^g_XA)Y,Z)= \frac{1}{3} dF(X,Y,Z)+\frac{1}{3}
dF(X,AY,AZ)-\frac16dF(AX,Y,AZ)-\frac16dF(AX,AY,Z).
\end{multline*}
\end{dfn}
On an almost paracontact NGT manifold the equalities \eqref{xi1AP}
hold, the Nijenhuis tensor is given by \eqref{skew1AP} and the
Nijenhuis tensor $N(AX,AY,AZ)$  is totally skew-symmetric
determined by
\begin{multline*}
N(AX,AY,AZ)=\frac13dF(AX,AY,Z)+\frac13dF(X,AY,AZ)+\frac13dF(AX,Y,AZ)+\frac13dF(X,Y,Z)\\
+2\eta(Z)d\eta(AX,Y)+2\eta(Y)d\eta(AZ,X)+2\eta(X)d\eta(AY,Z).
\end{multline*}

An example can be obtained as follows, start with a Nearly
paraK\"ahler manifold $N^{2n}$ one obtains on
$M^{2n+1}\times\mathbb R^+$ an almost paracontact NGT structure
taking $\eta=dt$.

{\bf Acknowledgments.} S.Ivanov is visiting University of
Pennsylvania, Philadelphia. S.I. thanks UPenn for providing the
support and an excellent research environment during the final
stage  of the paper. S.I. is partially supported by Contract DFNI
I02/4/12.12.2014 and  Contract 168/2014 with the Sofia University
"St.Kl.Ohridski". M.Z. was partially supported by and the project
EUROWEB, and by Serbian Ministry of Education, Science, and
Technological Development, Projects 174012...

\end{document}